\documentclass{amsart}
\usepackage[all]{xy}
\usepackage{color}
\usepackage{amsthm}
\usepackage{amssymb}
\usepackage[colorlinks=true]{hyperref}

\numberwithin{equation}{section}

\newtheorem{theorem}[equation]{Theorem}
\newtheorem*{theorem*}{Theorem}
\newtheorem{lemma}[equation]{Lemma}

\newtheorem*{conjecture*}{Mamma Conjecture}
\newtheorem*{conjecture1*}{Mamma Conjecture (revisited)}
\newtheorem{proposition}[equation]{Proposition}
\newtheorem{corollary}[equation]{Corollary}
\newtheorem*{corollary*}{Corollary}

\theoremstyle{remark}

\newtheorem{example}[equation]{Example}

\newtheorem{notation}[equation]{Notation}

\theoremstyle{remark}

\newcommand{\cA}{{\mathcal A}}
\newcommand{\cB}{{\mathcal B}}
\newcommand{\cC}{{\mathcal C}}
\newcommand{\cD}{{\mathcal D}}

\newcommand{\cM}{{\mathcal M}}

\newcommand{\cO}{{\mathcal O}}
\newcommand{\cP}{{\mathcal P}}

\newcommand{\cR}{{\mathcal R}}

\newcommand{\cT}{{\mathcal T}}
\newcommand{\cU}{{\mathcal U}}

\newcommand{\Spt}{\mathsf{Spt}}

\newcommand{\bbD}{\mathbb{D}}

\newcommand{\bbL}{\mathbb{L}}

\newcommand{\bbR}{\mathbb{R}}
\newcommand{\bbS}{\mathbb{S}}

\newcommand{\bbZ}{\mathbb{Z}}

\DeclareMathOperator{\Mot}{Mot}

\DeclareMathOperator{\Fun}{Fun} 
\DeclareMathOperator{\BK}{BKK} 

\newcommand{\uk}{\underline{k}}
\newcommand{\uA}{\underline{A}}

\newcommand{\dgcat}{\mathsf{dgcat}}
\newcommand{\perf}{\cD_\mathsf{perf}}
\newcommand{\perfdg}{\cD_\mathsf{perf}^{\dg}}

\newcommand{\dg}{\mathsf{dg}}

\newcommand{\uHom}{\underline{\mathsf{Hom}}}
\newcommand{\Hom}{\mathsf{Hom}}
\newcommand{\HomC}{\uHom_{\,!}}
\newcommand{\HomA}{\uHom_{\,\mathsf{A}}}

\newcommand{\rep}{\mathsf{rep}}
\newcommand{\Rep}{\mathsf{Rep}}
\newcommand{\repdg}{\mathsf{rep}_{\dg}}

\newcommand{\Ho}{\mathsf{Ho}}
\newcommand{\HO}{\mathsf{HO}}

\newcommand{\Hmo}{\mathsf{Hmo}}
\newcommand{\op}{\mathsf{op}}

\newcommand{\Map}{\mathsf{Map}}

\newcommand{\too}{\longrightarrow}

\newcommand{\ie}{\textsl{i.e.}\ }

\newcommand{\Madd}{\Mot_{\mathsf{A}}}
\newcommand{\Maddu}{\Mot^{\mathsf{u}}_{\mathsf{A}}}

\newcommand{\Uadd}{\cU_{\mathsf{A}}}
\newcommand{\Uaddu}{\cU^{\mathsf{u}}_{\mathsf{A}}}

\begin{document}

\title[Products, multiplicative Chern characters, and finite coefficients]{Products, multiplicative Chern characters, and finite coefficients via non-commutative motives}
\author{Gon{\c c}alo~Tabuada}

\address{Gon\c calo Tabuada, Departamento de Matematica, FCT-UNL, Quinta da Torre, 2829-516 Caparica,~Portugal }
\email{tabuada@fct.unl.pt}

\subjclass[2000]{19D45, 19D55, 18G55}
\date{\today}

\keywords{Algebraic $K$-theory, Chern characters, Non-commutative motives}

\thanks{The author was partially supported by the Clay Mathematics Institute, the Midwest Topology Network, and the Calouste Gulbenkian Foundation.}

\begin{abstract}
Products, multiplicative Chern characters, and finite coefficients, are unarguably among the most important tools in algebraic $K$-theory. Although they admit numerous different constructions, they are not yet fully understood at the conceptual level. In this article, making use of the theory of non-commutative motives, we change this state of affairs by characterizing these constructions in terms of simple, elegant, and precise universal properties. We illustrate the potential of our results by developing two of its manyfold consequences: (1) the multiplicativity of the negative Chern characters follows directly from a simple factorization of the mixed complex construction; (2) Kassel's bivariant Chern character admits an adequate extension, from the Grothendieck group level, to all higher algebraic K-theory.
\end{abstract}

\maketitle
\vskip-\baselineskip
\vskip-\baselineskip
\vskip-\baselineskip

\section{Introduction}

This article is part of a long-term project, whose objective is the study of (higher) algebraic $K$-theory via non-commutative motives. Previous contributions can be found in \cite{BGT,CT,CT1,Duke,IMRN,susp,transf,biv}. Here, we focus on products, multiplicative Chern characters, and finite coefficients.

Since the early days, it has been expected that algebraic $K$-theory would come equipped with some kind of {\em products}. Over the last decades such products were constructed by Loday, May, McCarthy, Milnor, Quillen, Waldhausen, and others in different settings and by making use of different tools~\cite{Loday-prod,May-prod,McCarthy,Milnor,Quillen-K, Wald,Wald-prod}. Among other important applications, they yielded new elements in algebraic $K$-theory and simplified proofs of Riemann-Roch theorems; see Weibel's survey~\cite{Weibel-survey}.

Algebraic $K$-theory is a very powerful and subtle invariant whose calculation is often out of reach. In order to capture some of its (multiplicative) information Connes-Karoubi, Dennis, Goodwillie, Hood-Jones, Kassel, McCarthy, and others constructed {\em (multiplicative) Chern characters} towards simpler theories by making use of a variety of highly involved techniques \cite{CK,Dennis,Goodwillie,HJ,Kassel-biv,McCarthy}.

In order to attack the Lichtenbaum-Quillen conjectures, Browder~\cite{Browder}  introduced the {\em algebraic $K$-theory with finite coefficients}. Since then, this flexible theory became one of the most important (working) tools in the field.

Although the importance of the aforementioned constructions extends well beyond algebraic $K$-theory, they remained rather {\em ad hoc} and somewhat mysterious. Therefore, a solution to the following question is of major importance:

\vspace{0.2cm}

\noindent\textbf{Question: }\textit{Is it possible to characterize all the aforementioned constructions by simple, elegant, and precise universal properties ?}

In this article we affirmatively answer this question. Moreover, the theory of {\em non-commutative motives} plays a central role is such characterizations.
 
\subsection*{Non-commutative motives}
A {\em differential graded (=dg) category}, over a fixed base commutative ring $k$, is a category enriched over
cochain complexes of $k$-modules; see \S\ref{sub:dg} for details. All the classical invariants such as cyclic homology (and its variants), algebraic $K$-theory, and even topological cyclic homology, extend naturally from $k$-algebras to dg categories. In order to study all these invariants simultaneously the notion of {\em additive invariant} was introduced in~\cite{Duke}. This notion, that we now recall, makes use of the language of Grothendieck derivators, a formalism which
allows us to state and prove precise universal properties; see \S\ref{sub:derivators}. The category $\dgcat$ of dg categories carries a Quillen model structure, whose weak equivalences are the {\em derived Morita equivalences} (see \S\ref{sub:Morita}), and so it gives rise to a derivator $\HO(\dgcat)$. Let $\bbD$ be a triangulated derivator and $\mathit{E}: \HO(\dgcat) \to \bbD$ a morphism of derivators. We say that $\mathit{E}$ is an {\em additive invariant} if it preserves filtered homotopy
colimits and sends split exact sequences (\ie sequences of dg categories which become split exact after passage to the associated derived categories; see \cite[Definition~13.1]{Duke}) to direct sums.
By the additivity results of Blumberg-Mandell, Keller, Waldhausen, and others~\cite{BM,cyclic,AGT,Wald} all the mentioned invariants give rise to additive invariants.
In~\cite{Duke} the {\em universal additive invariant} was constructed
\begin{equation*}
 \Uadd: \HO(\dgcat) \too \Madd\,.
\end{equation*}
Given any triangulated derivator $\bbD$ we have an induced equivalence of categories
\begin{equation}\label{eq:cat}
(\Uadd)^{\ast}:\ \HomC(\Madd, \bbD) \stackrel{\sim}{\too} \HomA(\HO(\dgcat), \bbD)\,,
\end{equation}
where the left-hand side denotes the category of homotopy colimit preserving morphisms of derivators and the right-hand side the category of additive invariants. Because of this universal property, which is reminiscent of the theory of motives, the derivator $\Madd$ is called the {\em additive motivator}, and its base category $\Madd(e)$ the {\em triangulated category of non-commutative motives}.

The tensor product extends naturally from $k$-algebras to dg categories, giving rise to a (derived) symmetric monoidal structure $-\otimes^\bbL-$ on $\HO(\dgcat)$ whose unit is the dg category $\uk$ with one object and with $k$ as the dg algebra of endomorphisms. In \cite{CT1} this monoidal structure was extended to $\Madd$ in a universal way\footnote{In \cite{CT1} we have considered the localizing analogue of $\Madd$. However, the arguments in the additive case are completely similar.}: given any triangulated derivator $\bbD$, endowed with a homotopy colimit preserving monoidal structure, the above equivalence~\eqref{eq:cat} admits a symmetric monoidal sharpening
\begin{equation}\label{eq:catmon}
(\Uadd)^{\ast}:\ \HomC^\otimes(\Madd, \bbD) \stackrel{\sim}{\too} \HomA^\otimes(\HO(\dgcat), \bbD)\,.
\end{equation}
As any triangulated derivator, $\Madd$ is naturally enriched over spectra; let us denote by $\bbR\Hom(-,-)$ this enrichment. In \cite{Duke} it was proved that algebraic $K$-theory not only descends uniquely to $\Madd$ (since it is an additive invariant) but, moreover, it becomes co-representable by the unit object, \ie for every dg category $\cA$ we have a natural equivalence of spectra
\begin{equation}\label{eq:corep-spt}
\bbR\Hom(\Uadd(\uk),\Uadd(\cA)) \simeq K(\cA)\,.
\end{equation}
In the triangulated category $\Madd(e)$ we have abelian group isomorphisms
\begin{equation}\label{eq:corep-K}
\Hom(\Uadd(\uk),\Uadd(\cA)[-n]) \simeq K_n(\cA) \qquad n \geq 0\,.
\end{equation}
\section{Statement of results}\label{sec:statements}
\subsection*{Products}
Let $\cA$ and $\cB$ be dg categories. On one hand, we can consider the associated categories $\widehat{\cA}$, $\widehat{\cB}$ and $\widehat{\cA\otimes \cB}$ of perfect modules (see Notation~\ref{not:hats}) and the bi-exact functor (see \S\ref{sub:pairings})
\begin{eqnarray}\label{eq:bi-exact}
\widehat{\cA} \times \widehat{\cB} \too \widehat{\cA \otimes \cB} && (M,N) \mapsto M \otimes_k N\,.
\end{eqnarray}
Following Waldhausen (see \S\ref{sub:Kdg}) we obtain then a pairing in algebraic $K$-theory
\begin{equation}\label{eq:pairing1}
K(\cA) \wedge K(\cB) \too K(\cA \otimes \cB) \,.
\end{equation}
On the other hand, if $\cA$ (or $\cB$) is {\em $k$-flat} (\ie for every pair $(x,y)$ of objects in $\cA$ the functor $\cA(x,y)\otimes-$ preserves quasi-isomorphisms; see \cite[\S4.2]{ICM}) $\cA \otimes^\bbL \cB \simeq \cA \otimes \cB $ and so the symmetric monoidal structure on $\Madd(e)$ combined with \eqref{eq:corep-spt} gives rise to another pairing in algebraic $K$-theory
\begin{equation}\label{eq:pairing3}
K(\cA) \wedge K(\cB) \too K(\cA \otimes^\bbL \cB)\simeq K(\cA \otimes \cB)\,.
\end{equation}
\begin{theorem}\label{thm:1}
The pairings \eqref{eq:pairing1} and \eqref{eq:pairing3} agree up to homotopy. In particular, the abelian group homomorphisms $K_i(\cA) \otimes_\bbZ K_j(\cB) \to K_{i+j}(\cA \otimes \cB)$ induced by the above pairing~\eqref{eq:pairing1} agree with the abelian group homomorphisms obtained by combining the symmetric monoidal structure on $\Madd(e)$ with \eqref{eq:corep-K}.
\end{theorem}
\begin{example}{(Commutative algebras)}\label{ex:commutative}
Let $\cA=\cB=\underline{A}$, with $A$ a $k$-flat commutative algebra. Since $A$ is commutative, its multiplication is a morphism of $k$-algebras and hence a dg functor. By composing \eqref{eq:pairing3} with the induced map
$$K(\underline{A}\otimes\underline{A}) \simeq \bbR\Hom(\Uadd(\uk), \Uadd(\underline{A}\otimes\underline{A})) \too \bbR\Hom(\Uadd(\uk), \Uadd(\underline{A})) \simeq K(\underline{A})$$
we recover the algebraic $K$-theory pairing on $A$ constructed by Waldhausen in \cite{Wald-prod}. In particular, we recover the (graded-commutative) multiplicative structure on $K_\ast(A)$ constructed originally by Loday in \cite{Loday-prod}; see \cite[\S4]{Weibel-survey} for the agreement between the approach of Waldhausen and the approach of Loday.
\end{example}

\begin{example}{(Schemes)}\label{ex:schemes}
Let $X$ be a quasi-compact and quasi-separated scheme over a base field $k$. Recall from \cite{LO} (or from \cite[Example~4.5(i)]{CT1}) that the derived category of perfect complexes of $\cO_X$-modules admits a natural differential graded (=dg) enhancement $\perf^\dg(X)$. Under this enhancement, the bi-exact functor $(E^\cdot, F^\cdot) \mapsto E^\cdot \otimes_{\cO_X}^\bbL F^\cdot$ lifts to a dg functor $\perf^\dg(X) \otimes \perf^\dg(X) \to \perf^\dg(X)$.
Therefore, since the algebraic $K$-theory of $X$ can be recovered from $\perf^\dg(X)$ (see \cite[\S5.2]{ICM}), the composition of~\eqref{eq:pairing3} with the induced map
$$\bbR\Hom(\Uadd(\uk), \Uadd(\perf^\dg(X)\otimes\perf^\dg(X))) \too \bbR\Hom(\Uadd(\uk), \Uadd(\perf^\dg(X)))$$
gives rise to the algebraic $K$-theory pairing on $X$ constructed originally by Thomason-Trobaugh in \cite[\S3.15]{TT}.
\end{example}
Note that Theorem~\ref{thm:1} (and Examples \ref{ex:commutative}\text{-}\ref{ex:schemes}) offers an elegant conceptual characterization of the algebraic $K$-theory products. Informally speaking, algebraic $K$-theory is the additive invariant co-represented by the unit of $\Madd(e)$, and the algebraic $K$-theory products are the operations naturally induced by the symmetric monoidal structure of $\Madd(e)$. Work in progress with Blumberg and Gepner suggests that a stronger result, in the setting of (stable) infinity categories, can be proved: algebraic $K$-theory carries a {\em unique} $E_{\infty}$-multiplicative structure. 
\subsection*{Multiplicative Chern characters}
Let $E: \HO(\dgcat) \to \bbD$ be a symmetric monoidal additive invariant. Thanks to equivalence \eqref{eq:catmon} there is a unique symmetric monoidal morphism of derivators $\overline{E}$ making the diagram
$$
\xymatrix{
\HO(\dgcat) \ar[r]^-E \ar[d]_{\Uadd} & \bbD \\
\Madd \ar[ur]_-{\overline{E}} &
}
$$
commute. Let us write ${\bf 1}:=E(\underline{k})$ for the unit of $\bbD(e)$. Then, for every dg category $\cA$ the morphism $\overline{E}$ combined with equivalence \eqref{eq:corep-spt} gives rise to a canonical map 
\begin{equation*}
ch : K(\cA) \simeq \bbR\Hom(\Uadd(\uk), \Uadd(\cA)) \too \bbR\Hom_{\bbD(e)}({\bf 1}, E(\cA))\,.
\end{equation*}
Now, recall from~\cite[Examples~7.9-7.10]{CT1} that the mixed complex $(C)$ and the Hochschild homology $(HH)$ constructions give rise to symmetric monoidal additive invariants
\begin{eqnarray*}
C:\HO(\dgcat) \too \HO(\cC(\Lambda)) && HH: \HO(\dgcat) \too \HO(\cC(k))\,.
\end{eqnarray*} 
Here, $\Lambda$ is the dg algebra $k[\epsilon]/\epsilon^2$ with $\epsilon$ of degree $-1$ and $d(\epsilon)=0$. Moreover, there is a symmetric monoidal forgetful morphism $\HO(\cC(\Lambda)) \to \HO(\cC(k))$ whose pre-composition with $C$ equals $HH$. By the above considerations we obtain then a canonical commutative diagram of spectra
\begin{equation}\label{eq:diag-spt}
\xymatrix{
& \bbR\Hom(k,C(\cA)) \ar[d] \\
K(\cA) \ar[ur]^-{ch^{\Lambda}} \ar[r]_-{ch^k} & \bbR\Hom(k, HH(\cA))\,.
}
\end{equation}
As shown in \cite[Examples 8.9-8.10]{CT1} we have natural isomorphisms of abelian groups
\begin{eqnarray*}
 \Hom_{\cD(\Lambda)}(k,C(\cA)[-n]) \simeq  HC_n^{-}(\cA) && \Hom_{\cD(k)}(k,HH(\cA)[-n]) \simeq  HH_n(\cA)\,,
\end{eqnarray*}
where $HC_n^{-}(\cA)$ denotes the $n^{\textrm{th}}$ negative cyclic homology group of $\cA$. Therefore, by passing to the homotopy groups in the above diagram \eqref{eq:diag-spt} (or equivalently by considering the morphisms in the triangulated categories $\Madd(e)$, $\cD(\Lambda)$ and $\cD(k)$) we obtain canonical commutative diagrams of abelian groups
\begin{equation*}
\xymatrix{
& HC_n^{-}(\cA) \ar[d] \\
K_n(\cA) \ar[ur]^-{ch_n^\Lambda} \ar[r]_-{ch^k_n} & HH_n(\cA)\,.
}
\end{equation*}
\begin{theorem}\label{thm:2}
When $\cA=\underline{A}$, with $A$ a $k$-algebra, the above abelian group homomorphisms $ch^\Lambda_n$ and $ch^k_n$ agree, respectively, with the $n^{\textrm{th}}$ negative Chern character independently constructed by Hood-Jones~\cite{HJ} and by Goodwillie~\cite{Goodwillie}, and with the $n^{\textrm{th}}$ Dennis trace map originally constructed by Dennis~\cite{Dennis}.
\end{theorem}
Theorem~\ref{thm:2} provides a simple conceptual characterization of the highly involved work of Dennis, Goodwillie, and Hood-Jones. Intuitively it tell us that the negative Chern character, resp. the Dennis trace map, is the unique symmetric monoidal factorization of the mixed complex construction, resp. of the Hochschild homology construction, through the universal additive invariant.

When the $k$-algebra $A$ is commutative, the multiplicative structures on $HC^-_\ast(A)$ and $HH_\ast(A)$ can be recovered, respectively, from the symmetric monoidal structures of $\cD(\Lambda)$ and $\cD(k)$; see \cite[\S5]{JK}. Therefore, by combining Theorem~\ref{thm:2} with Example~\ref{ex:commutative}, we obtain {\em for free} the following result (proved originally by Hood-Jones~\cite[\S5]{HJ} using a highly elaborate construction), which was one of the main driving forces behind the development of negative cyclic homology.
\begin{corollary}
Let $A$ be a commutative $k$-algebra. Then, the negative Chern characters and the Dennis trace maps are multiplicative.
\end{corollary}
\subsection*{Bivariant picture}
Jones and Kassel, by drawing inspiration from Kasparov's $KK$-theory, introduced in \cite{JK} the {\em bivariant cyclic cohomology theory} of $k$-algebras. 
Immediately afterwards, Kassel~\cite{Kassel-biv} introduced the {\em bivariant algebraic $K$-theory} and constructed a (multiplicative) bivariant Chern character from it to bivariant cyclic cohomology. These involved constructions can be compactly  expressed in terms of a symmetric monoidal functor
$$ch_B: \BK \too \cD(\Lambda)\,.$$  
Roughly, $\BK$ is the additive category\footnote{The initials $\BK$ stand for bivariant, $K$-theory, and Kassel respectively.} whose objects are the $k$-algebras and whose (abelian group) morphisms $\Hom_{\BK}(B,A)$ are the Grothendieck groups of the exact categories of those $B\text{-}A$-bimodules which are projective and of finite type as $A$-modules; for details see the proof of Theorem~\ref{thm:3}.

Kassel's bivariant Chern character was quite an important contribution. For example, it yielded a simple proof of the invariance of cyclic homology under Morita equivalences. However, it has a serious drawback: it is only defined at the level of the Grothendieck groups. Kassel claimed in \cite[pages~368-369]{Kassel-biv} that an extension to the higher algebraic $K$-theory groups should exist but, to the best of the author's knowledge, this important problem remained wide open during the last decades. The following result changes this state of affairs.
\begin{theorem}\label{thm:3}
There is a natural symmetric monoidal additive functor $N$ making the following diagram
\begin{equation*}
\xymatrix{
\BK \ar[d]_N \ar[r]^-{ch_B} & \cD(\Lambda) \\
\Madd(e) \ar[ur]_{\overline{C}(e)} & 
}
\end{equation*}
commute. Moreover, given $k$-algebras $A$ and $B$, with $B$ {\em homotopically finitely presented} (the homotopical version of the classical notion of finite presentation; see \cite[\S4.7]{ICM}), we have an induced isomorphism
\begin{equation}\label{isom-thm3}
\Hom_{\BK}(B,A) \stackrel{\sim}{\too} \Hom_{\Madd(e)}(N(B),N(A))\,.
\end{equation}
\end{theorem}
As proved in \cite{Duke} (see \eqref{eq:biv-ext1}\textrm{-}\eqref{eq:biv-ext2}), equivalence \eqref{eq:corep-spt} and isomorphisms \eqref{eq:corep-K} admit natural bivariant extensions. Therefore, Theorem~\ref{thm:3} (combined with Theorem~\ref{thm:2}) shows us that $\Madd(e)$ and $\overline{C}(e)$ should be considered as the correct ``higher-dimensional extensions'' of the additive category $\BK$ and of the bivariant Chern character $ch_B$, respectively. Moreover, in contrast with Kassel's {\em ad hoc} construction, the triangulated category $\Madd(e)$ and the (symmetric monoidal) triangulated functor $\overline{C}(e)$ are characterized by simple and precise universal properties.
\subsection*{Finite coefficients}
In this subsection we assume for simplicity that $k=\bbZ$. Given an object $X$ in a triangulated category $\cT$ and an integer $l\geq 2 $, we write $\cdot l$ for the $l$-fold multiple of the identity morphism in the abelian group $\Hom_{\cT}(X,X)$. Following the topological convention, we define the {\em mod-$l$ Moore object $X/l$ of $X$} as the cone of $\cdot l$, \ie as the object which is part of a distinguished triangle
\begin{equation*}
X \stackrel{\cdot l}{\too} X \too X/l \too X[1]\,.
\end{equation*}
\begin{proposition}\label{prop:4} Let $l\geq 2$ be an integer. Then, for every dg category $\cA$ we have an equivalence of spectra
\begin{equation*}
\bbR \Hom\left(\Uadd(\underline{\bbZ})/l, \Uadd(\cA) \right) \simeq K(\cA;\bbZ/l)[-1]\,,
\end{equation*}
where $K(\cA;\bbZ/l)$ denotes Browder's mod-$l$ algebraic $K$-theory spectrum of $\cA$ (see \S\ref{sub:Kdg}). Moreover, if $l=pq$ with $p$ and $q$ coprime integers, we have an isomorphism $\Uadd(\underline{\bbZ})/l \simeq \Uadd(\underline{\bbZ})/p \oplus \Uadd(\underline{\bbZ})/q$ in $\Madd(e)$.
\end{proposition}
Making use of the above triangle (with $X = \Uadd(\underline{\bbZ})$) and of isomorphisms \eqref{eq:corep-K} we recover then Browder's long exact sequence relating $K$-theory with mod-$l$ $K$-theory
$$ \cdots \to K_{n}(\cA) \stackrel{\cdot l}{\to} K_n(\cA) \to K_n(\cA;\bbZ/l) \to K_{n-1}(\cA) \stackrel{\cdot l}{\to} K_{n-1}(\cA) \to \cdots\,.$$
Proposition~\ref{prop:4} gives a precise conceptual characterization of Browder's construction. Roughly speaking, mod-$l$ algebraic $K$-theory is the additive invariant co-represented by the mod-$l$ Moore object of the unit $\Uadd(\uk)$ of $\Madd(e)$. Moreover, Proposition~\ref{prop:4} shows us that we can always assume that $l$ is a prime power.
\begin{example}{(Rings and Schemes)}
Let $l$ be a prime number and $l^\nu$ a power of it. When $\cA=\underline{A}$, with $A$ a ring, the spectrum $K(\underline{A}; \bbZ/l^\nu)$ agrees with the one $K(A; \bbZ/l^\nu)$ constructed originally by Browder in \cite{Browder}. When $\cA= \perfdg(X)$, with $X$ a quasi-compact and quasi-separated scheme (see Example~\ref{ex:schemes}), the spectrum $K(\perf^\dg(X);\bbZ/l^\nu)$ agrees with the one $K(X;\bbZ/l^\nu)$ constructed originally by Thomason-Trobaugh in \cite[\S9.3]{TT}.
\end{example}

\medbreak

\medbreak

\noindent\textbf{Acknowledgments\,:} The author is very grateful to Paul Balmer, Alexander Beilinson, Andrew Blumberg, Guillermo Corti{\~n}as, David Gepner, Christian Haesemeyer, Bernhard Keller, and Randy McCarthy for stimulating conversations. He would like also to thank the departments of mathematics of UCLA and UIC for its hospitality and the Clay Mathematics Institute, the Midwest Topology Network, and the Calouste Gulbenkian Foundation for financial support.
\section{Preliminaries}\label{sec:prelim}
\subsection{Notations}\label{sub:notation}
We will work over a fixed commutative base ring $k$. The category of finite ordered sets and non-decreasing monotone maps will be denoted by $\Delta$. Given an integer $r \geq 1$, we will write $\Delta^{(r)}$ for the product of $\Delta$ with itself $r$ times. The geometric realization of a (multi-)simplicial set and of a (multi-simplicial) category (obtained by first passing ot the nerve $N.$) will be denoted simply by $|-|$. We will write $\Spt$ for the classical category of spectra~\cite{Adams}.
Given a model category $\cM$ in the sense of Quillen~\cite{Quillen} we will write $\Ho(\cM)$ for its homotopy category and $\Map(-,-)$ for its homotopy function complex; see~\cite[Definition~17.4.1]{Hirschhorn}.  Given categories $\cC$ and $\cD$, we will denote by $\Fun(\cC,\cD)$ the category of functors with natural transformations as morphisms. Finally, the adjunctions will be displayed vertically with the left (resp. right) adjoint on the left- (resp. right-) hand side. 

\subsection{Grothendieck derivators}\label{sub:derivators}
We will use the basic language of Grothendieck derivators~\cite{Grothendieck} which can be easily acquired by skimming through \cite[\S1]{CN} or \cite[Appendix~A]{CT,CT1}. An arbitrary derivator will be denoted by $\bbD$. The essential example to keep in mind is the (triangulated) derivator $\bbD=\HO(\cM)$
associated to a (stable) Quillen model category~$\cM$
and defined for every small category $I$ by $$\HO(\cM)(I):=\Ho(\Fun(I^\op,\cM))\,.$$
We will denote by $e$ the $1$-point category with one object and one
identity morphism. Heuristically, the category $\bbD(e)$ is the
basic ``derived" category under consideration in the
derivator~$\bbD$. For instance, if $\bbD=\HO(\cM)$ then
$\bbD(e)$ is the homotopy category $\Ho(\cM)$. As shown in \cite[\S A.3]{CT}, every triangulated derivator $\bbD$ is canonically
enriched over spectra. We will denote by $\bbR\Hom_{\bbD(e)}(X,Y)$
the spectrum of maps from $X$ to $Y$ in $\bbD(e)$. When there is no ambiguity, we
will write simply $\bbR\Hom(X,Y)$.

\section{Waldhausen's constructions}\label{appendix}
In this section we recall some of Waldhausen's \cite{Wald} foundational constructions. This will give us the occasion to fix important notations which will allow us to greatly simplify the proofs of Theorems~\ref{thm:1}, \ref{thm:2}, \ref{thm:3} and of Proposition~\ref{prop:4}.
\subsection{$K$-theory spectrum}\label{sub:spectrum}
Given a non-negative integer $q$ we will denote by $[q]$ the ordered set $\{0<1<\cdots <q\}$ considered as a category.  We will write $\mathrm{Ar}[q]$ for the category of arrows in $[q]$ (\ie the functor category $\Fun([1],[q])$) and $(i/j)$ for the arrow from $i$ to $j$ with $i \leq j$. Let $\cC$ be a category with cofibrations and weak equivalences in the sense of Waldhausen~\cite[\S1.2]{Wald}. We will denote by $S_{q}\cC$ the full subcategory of $\Fun(\mathrm{Ar}[q], \cC)$ whose objects are the functors $A: (i/j) \mapsto A_{i,j}$ such that $A_{j,j}=\ast$, and for every triple $i \leq j \leq k$ the following diagram
$$
\xymatrix{
*+<3.0ex>{A_{i,j}} \ar@{>->}[r] \ar[d] \ar@{}[dr]|{\lrcorner}& A_{i,k} \ar@{->>}[d] \\
\ast \simeq A_{j,j} \ar[r] & A_{j,k}
}
$$
is a co-cartesian. As mentioned by Waldhausen in \cite[page~328]{Wald} an object $A \in S_q\cC$ corresponds to a sequence of cofibrations $(\ast=A_{0,0} \rightarrowtail A_{0,1} \rightarrowtail \cdots \rightarrowtail A_{0,q})$ in $\cC$ together with a choice of subquocients $A_{i,j}=A_{0,j}/A_{0,i}$. Moreover, $S_q\cC$ can be considered as a category with cofibrations and weak equivalences in a natural way; see \cite[page~329]{Wald}. By letting $q$ vary we obtain a simplicial category with cofibrations and weak equivalences $S.\cC$ and by iterating the $S.$-construction we obtain multi-simplicial categories with cofibrations and weak equivalences $S.^{(n)} \cC:= \underbrace{S. \cdots S. C}_{n\, \mathrm{times}}$. Given non-negative integers $q_1, \ldots , q_n$, the category $S_{q_1}\cdots S_{q_n} \cC$ will be denoted by $S^{(n)}_{q_1, \ldots, q_n}\cC$. By passing to the geometric realization of the multi-simplicial categories  $w S.^{(n)}\cC$ of weak equivalences we obtain finally the {\em algebraic $K$-theory spectrum}:
\begin{equation}\label{eq:spectrum}
K(\cC):\quad  \Omega |w S.\cC| \quad |wS.\cC| \quad |wS.^{(2)}\cC| \quad \cdots \quad  |wS.^{(n)}\cC| \quad \cdots \,.
\end{equation}
As shown in \cite[page~330]{Wald} this is a connective $\Omega$-spectrum with structure maps $S^1 \wedge |wS.^{(n)}\cC| \to |wS.^{(n+1)}\cC|$ induced by the natural identification $\cC \simeq S_1\cC$. 
\subsection{Pairings}\label{sub:pairings}
Let $\cA$, $\cB$ and $\cC$ be categories with cofibrations and weak equivalences and $T: \cA \times \cB \to \cC$ a {\em bi-exact functor} in the sense of Waldhausen~\cite[page~342]{Wald}, \ie a functor satisfying the following two conditions: for every pair of object $A \in \cA$ and $B \in \cB$, the induced functors $T(A,-)$ and $T(-,B)$ are exact; for every pair of cofibrations $A \rightarrowtail A'$ and $B \rightarrowtail B'$ in $\cA$ and $\cB$, respectively, the induced map $T(A',B) \amalg_{T(A,B)} T(A,B') \rightarrowtail T(A',B')$ is a cofibration in $\cC$. Given non-negative integers $q$ and $p$, we obtain then an induced bi-exact functor $S_q \cA \times S_p \cB \to S^{(2)}_{q,p}\cC$ which sends the pair
$$ \left((\ast=A_{0,0} \rightarrowtail A_{0,1} \rightarrowtail \cdots \rightarrowtail A_{0,q}), (\ast=B_{0,0} \rightarrowtail B_{0,1} \rightarrowtail \cdots \rightarrowtail B_{0,p})\right) \in S_q \cA \times S_p \cB$$
to the element of $S^{(2)}_{q, p}\cC$ represented by 
$$
\xymatrix@C=2em@R=1.1em{
*+<2.5ex>{T(A_{0,0}, B_{0,0})} \ar@{>->}[d] \ar@{>->}[r] &  *+<2.5ex>{T(A_{0,0},B_{0,1})} \ar@{>->}[r] \ar@{>->}[d]&*+<3.5ex> {\cdots} \ar@{>->}[r] & *+<2.5ex>{T(A_{0,0},B_{0,p})} \ar@{>->}[d] \\
*+<2.5ex>{T(A_{0,1}, B_{0,0})} \ar@{>->}[d] \ar@{>->}[r] &  *+<2.5ex>{T(A_{0,1},B_{0,1})} \ar@{>->}[r] \ar@{>->}[d]&*+<3.5ex>{\cdots} \ar@{>->}[r] & *+<2.5ex>{T(A_{0,1},B_{0,p})} \ar@{>->}[d] \\
*+<3.5ex>{\vdots} \ar@{>->}[d] & *+<3.5ex>{\vdots} \ar@{>->}[d] & *+<3.5ex>{\ddots} & *+<3.5ex>{\vdots} \ar@{>->}[d]  \\
*+<2.5ex>{T(A_{0,q}, B_{0,0})} \ar@{>->}[r] &  *+<2.5ex>{T(A_{0,q},B_{0,1})} \ar@{>->}[r] &*+<3.5ex>{\cdots} \ar@{>->}[r] & *+<3.5ex>{T(A_{0,q},B_{0,p})}\,,
}
$$
with associated subquotients $T(A_{0,j},B_{0,l})/T(A_{0,i},B_{0,k}):= T(A_{j/i}, B_{l/k})$. The iteration of this construction furnish us bi-exact functors
\begin{equation*}
S^{(n)}_{q_1,\ldots,q_n} \cA \times S^{(m)}_{p_1, \ldots, p_m} \cB \too S^{(n+m)}_{\substack{q_1,\ldots, q_n\\p_1, \ldots, p_m}} \cC\,,
\end{equation*}
where $q_1, \ldots, q_n$ and $p_1, \ldots, p_m$ are non-negative integers. By passing to the geometric realization of the multi-simplicial categories of weak equivalences we obtain induced maps
\begin{equation*}
| wS.^{(n)}\cA| \times | wS.^{(m)}\cB | \too | wS.^{(n+m)}\cC|
\end{equation*}
which factor through $| wS.^{(n)}\cA| \wedge | wS.^{(m)}\cB |$. Finally, due to the $\Omega$-structure of the algebraic $K$-theory spectrum, these latter maps assemble themselves in a well-defined algebraic $K$-theory pairing 
$$ K(\cA) \wedge K(\cB) \too K(\cC)\,.$$

\section{Differential graded categories}\label{sub:dg}
In this section we collect the notions and results concerning dg categories which will be used in the proofs of Theorems~\ref{thm:1}, \ref{thm:2}, \ref{thm:3} and of Proposition~\ref{prop:4}.

We will denote by $\cC(k)$ be the category of (unbounded) complexes of $k$-modules. We will use co-homological notation, \ie the differential increases the degree. A {\em differential graded (=dg) category} is a category enriched over $\cC(k)$ (morphisms sets are complexes) in such a way that composition fulfills the Leibniz rule\,: $d(f\circ g)=(df)\circ g+(-1)^{\textrm{deg}(f)}f\circ(dg)$. Given a $k$-algebra $A$, we will write $\underline{A}$ for the dg category with one object and with $A$ as the dg algebra of endomorphisms (concentrated in degree zero). For a survey article, we invite the reader to consult Keller's ICM adress~\cite{ICM}.
\subsection{Dg (bi-)modules}\label{sub:modules}
Let $\cA$ be a dg category. The {\em opposite dg category} $\mathcal{A}^\op$ has the same objects as $\cA$ and complexes of morphisms given by
$\mathcal{A}^\op(x,y):=\mathcal{A}(y,x)$. 
Recall from~\cite[\S3.1]{ICM} that a {\em
right dg $\cA$-module} (or simply a $\cA$-module) is a dg functor $M: \cA^\op \rightarrow
\cC_{\dg}(k)$ with values in the dg category $\cC_{\dg}(k)$ of
complexes of $k$-modules. We will denote by $\cC(\cA)$ (resp.\ by
$\cC_{\dg}(\cA)$) the category (resp.\ dg category) of $\cA$-modules. Recall from~\cite[Theorem~3.2]{ICM} that $\cC(\cA)$ carries a standard projective model structure whose weak equivalences are the quasi-isomorphisms. The {\em derived category $\cD(\cA)$ of $\cA$} is the localization of $\cC(\cA)$ with respect to the class of quasi-isomorphisms.
\begin{notation}\label{not:hats}
In order to simplify the exposition, we will denote by $\widehat{\cA}$ (resp. by $\widehat{\cA}_\dg$) be the full subcategory of $\cC(\cA)$ (resp. of $\cC_\dg(\cA)$) formed by the $\cA$-modules which are cofibrant and that become compact in $\cD(\cA)$. As shown in \cite[\S3]{DS} the category $\widehat{\cA}$, endowed with the cofibrations and weak equivalences of the projective model structure, is a category with cofibrations and weak equivalences in the sense of Waldhausen~\cite{Wald}. The objects of $\widehat{\cA}$ (and of $\widehat{\cA}_\dg$) will be called the {\em perfect} $\cA$-modules. 
\end{notation}

Recall from \cite[\S2.3]{ICM} that the {\em tensor product} $\cA \otimes \cB$ of two dg categories is defined as follows: the set of objects is the cartesian product of the set of objects of $\cA$ and $\cB$ and $(\cA \otimes \cB)((x,z),(y,w)):= \cA(x,y) \otimes \cA(z,w)$. By a {\em right dg $\cA\text{-}\cB$-bimodule} (or simply a $\cA\text{-}\cB$-bimodule) we mean a dg functor $\cA^\op \otimes \cB \to \cC_\dg(k)$.
\subsection{Derived Morita equivalences}\label{sub:Morita}
A dg functor $F:\cA \to \cB$ is called a {\em derived Morita equivalence} if it induces an equivalence $\cD(\cB) \stackrel{\sim}{\to} \cD(\cA)$ of triangulated categories; see \cite[\S4.6]{ICM}. Recall from~\cite[Theorem~4.10]{ICM} that $\dgcat$ carries a Quillen model structure, whose weak equivalences are the {\em derived Morita equivalences}.
We will denote by $\Hmo$ the homotopy category hence obtained. 
The tensor product of dg categories can be naturally derived $-\otimes^\bbL-$, thus giving rise to a symmetric monoidal structure on $\Hmo$. Given dg categories $\cA$ and $\cB$, let $\rep(\cB,\cA)$ be the full triangulated subcategory of $\cD(\cB^\op \otimes^\bbL \cA)$ spanned by the cofibrant  $\cB\text{-}\cA$-bimodules $X$ such that for every object $x \in \cB$ the associated $\cA$-module $X(x,-)$ becomes compact in $\cD(\cA)$. Recall from \cite[\S4.2 and \S4.6]{ICM} that there is a natural bijection between $\Hom_{\Hmo}(\cB,\cA)$ and the isomorphism classes of objects in $\rep(\cB,\cA)$. Moreover, composition in $\Hmo$ corresponds to the (derived) tensor product of bimodules. Let $\cR(\cB,\cA)$ be the subcategory of $\cC(\cB^\op\otimes^\bbL \cA)$ with the same objects as $\rep(\cB,\cA)$ and whose morphisms are the quasi-isomorphisms. As explained in \cite[\S4]{ICM} there is a canonical weak equivalence of simplicial sets between $\Map_\Hmo(\cB,\cA)$ and the nerve of the category $\cR(\cB,\cA)$. Finally, recall from \cite[\S4.3]{ICM} that the (derived) symmetric monoidal structure on $\Hmo$ is closed. Given dg categories $\cA$ and $\cB$, the internal Hom-functor $\repdg(\cB,\cA)$ is the full dg subcategory of $\cC_\dg(\cB^\op \otimes^\bbL \cA)$ with the same objects as $\rep(\cB,\cA)$.
\subsection{Algebraic K-theories}\label{sub:Kdg}
Let $\cA$ be a dg category. Recall from \cite[\S5.2]{ICM} that the {\em algebraic $K$-theory spectrum $K(\cA)$ of $\cA$} is the spectrum \eqref{eq:spectrum} associated to the category with cofibrations and weak equivalences $\widehat{\cA}$. Now, assume that $k=\bbZ$ and consider the distinguished triangle $\bbS \stackrel{\cdot l}{\to} \bbS \to \bbS /l \to \bbS[1]$ in $\Ho(\Spt)$, where $l\geq 2$ is an integer, $\bbS$ is the sphere spectrum, and $\cdot l$ is the $l$-fold multiple of the identity morphism. Following Browder~\cite{Browder}, the {\em mod-$l$ algebraic $K$-theory spectrum $K(\cA;\bbZ/l)$ of $\cA$} is the derived smash product spectrum $\bbS/l\wedge^\bbL K(\cA)$.

\section{Proof of Theorem~\ref{thm:1}}
\begin{proof}
The proof will consist on showing that the maps
\begin{equation}\label{eq:maps}
K(\cA)_n \times K(\cB)_m \too K(\cA \otimes \cB)_{n+m}\qquad n, m \geq 0
\end{equation}
associated to the pairings \eqref{eq:pairing1} and \eqref{eq:pairing3} agree up to homotopy; note that \eqref{eq:pairing1} and \eqref{eq:pairing3} are naturally defined on $K(\cA) \times K(\cB)$. We start by observing that it suffices to study the cases $n, m \geq 1$. Given an arbitrary dg category $\cC$ and an integer $r\geq 0$, let us write $K(\cC)[r]$ for the spectrum $K(\cC)[r]_n:= K(\cC)_{n+r}$ with the naturally induced structure maps. As explained in \S\ref{sub:pairings}, \eqref{eq:pairing1} is obtained by looping twice the pairing
$$ K(\cA)[1] \wedge K(\cB)[1] \too K(\cA\otimes \cB)[2]$$
induced by the bi-exact functor \eqref{eq:bi-exact}. In what concerns \eqref{eq:pairing3} the same phenomenon holds: the category $\Madd(e)$ is triangulated and its monoidal structure is homotopy colimit preserving, which implies that \eqref{eq:pairing3} is obtained by looping twice the pairing
$$ \bbR\Hom(\Uadd(\uk), \Uadd(\cA)[1]) \wedge \bbR\Hom(\Uadd(\uk), \Uadd(\cB)[1]) \too \bbR\Hom(\Uadd(\uk), \Uadd(\cA\otimes \cB)[2])\,.$$
Therefore, we can assume that $n$ and $m$ are (fixed) integers $\geq 1$. Now, recall from \S\ref{sub:Kdg} that we have the equalities:
\begin{eqnarray*}
K(\cA)_n= |wS.^{(n)}\widehat{\cA}| & K(\cB)_m= |wS.^{(m)}\widehat{\cB}| & K(\cA\otimes \cB)_{n+m}= |wS.^{(n+m)}\widehat{(\cA\otimes \cB)}|\,.
\end{eqnarray*} 
Following \S\ref{sub:pairings}, the maps \eqref{eq:maps} associated to the pairing \eqref{eq:pairing1} correspond to the ones
\begin{equation*}
|wS.^{(n)}\widehat{\cA}| \times |wS.^{(m)}\widehat{\cB}| \too |wS.^{(n+m)}\widehat{(\cA\otimes \cB)}|
\end{equation*} 
induced by the bi-exact functor \eqref{eq:bi-exact}. By construction, these latter maps can be furthermore expressed as the following homotopy colimit
\begin{equation}\label{eq:colimit}
\underset{\substack{ q_1, \ldots, q_n\\p_1,\ldots , p_m}}{\mathrm{hocolim}} \left(|wS^{(n)}_{q_1,\ldots, q_n}\widehat{\cA}| \times |wS^{(m)}_{p_1,\ldots, p_m}\widehat{\cB}| \too |wS^{(n+m)}_{\substack{q_1,\ldots, q_n\\p_1, \ldots, p_m}}\widehat{(\cA\otimes \cB)}|  \right)
\end{equation}
induced by the bi-exact functors
\begin{equation}\label{eq:k-exact}
S^{(n)}_{q_1,\ldots, q_n}\widehat{\cA} \times S^{(m)}_{p_1,\ldots, p_m}\widehat{\cB} \too S^{(n+m)}_{\substack{q_1,\ldots, q_n\\p_1, \ldots, p_m}}\widehat{(\cA\otimes \cB)}\,,
\end{equation}
where $q_1, \ldots, q_n$ and $p_1, \ldots, p_n$ are non-negative integers. The remainder of the proof consists of showing that the maps \eqref{eq:maps} associated to the pairing \eqref{eq:pairing3} can also be expressed as the above homotopy colimit \eqref{eq:colimit}. We start by observing that the bi-exact functors \eqref{eq:k-exact} admit a natural dg enrichment. Given an abstract dg category $\cC$, the dg enrichment $\cC_\dg(k)$ of $\cC(k)$ gives rise to a dg enrichment $\widehat{\cC}_\dg$ of $\widehat{\cC}$; see Notation~\ref{not:hats}. By construction, $S_q\widehat{\cC}$ and all the maps used in the $S.$-construction inherit a natural dg enrichment from $\widehat{\cC}_\dg$. In sum, we obtain well-defined multi-simplicial dg categories $S.^{(n)}\widehat{\cC}_\dg$. Under this dg enrichment the above bi-exact functors \eqref{eq:k-exact} become $k$-bilinear in the differential graded sense. Therefore, they give rise to well-defined dg functors
\begin{equation}\label{eq:keymaps}
S^{(n)}_{q_1, \ldots , q_n} \widehat{\cA}_\dg \otimes S^{(m)}_{p_1, \ldots , p_m} \widehat{\cB}_\dg \too S^{(n+m)}_{\substack{q_1, \ldots, q_n\\ p_1, \ldots, p_m}} \widehat{(\cA\otimes \cB)}_\dg
\end{equation}
which assemble themselves in a morphism
\begin{equation}\label{eq:keymap}
\Phi^{(n+m)}: S.^{(n)}\widehat{\cA}_\dg \otimes S.^{(m)}\widehat{\cB}_\dg \too S.^{(n+m)}\widehat{(\cA\otimes \cB)}_\dg
\end{equation} 
in $\HO(\dgcat)(\Delta^{(n+m)})$. Now, recall from \cite[\S15]{Duke} that the universal additive invariant is defined as the following composition
\begin{equation*}
\Uadd: \HO(\dgcat) \stackrel{\Uaddu}{\too} \Maddu \stackrel{\varphi}{\too} \Madd\,.
\end{equation*}
The derivator $\Maddu$ is the {\em unstable} analogue of $\Madd$ and the morphism $\varphi$ corresponds to the stabilization procedure. Let us write $\Sigma$ for the suspension functor in the pointed base category $\Maddu(e)$. Recall from \cite[\S1]{CN} that since $\Maddu$ is a derivator, we have the following adjunction
$$
\xymatrix{
\Maddu(\Delta^{(r)}) \ar@<-1ex>[d]_{\pi_!}\\
\Maddu(e) \ar@<-1ex>[u]_{\pi^\ast}
}
$$
associated to the projection functor $\pi: \Delta^{(r)} \to e$. 
\begin{lemma}\label{lem:aux2}
Let $\cC$ be a dg category. Then, for every integer $r \geq 1$ we have a natural isomorphism
$ \pi_!\, \Uaddu(S.^{(r)}\widehat{\cC}_\dg) \simeq \Sigma^{(r)} \Uaddu(\widehat{\cC}_\dg)$.
\end{lemma}
\begin{proof}
The proof goes by induction on $r$. The case $r=1$ was proved in \cite[Proposition~14.11]{Duke}. Now, suppose we have a natural isomorphism
\begin{equation}\label{eq:hyp}
\pi_!\, \Uaddu(S.^{(r)}\widehat{\cC}_\dg) \simeq \Sigma^{(r)} \Uaddu (\widehat{\cC}_\dg)\,.
\end{equation}
Following McCarthy~\cite[\S3.3]{McCarthy}, we consider the sequence of morphisms
$$ S.^{(r)} \widehat{\cC}_\dg \too PS.^{(r+1)}\widehat{\cC}_\dg \too S.^{(r+1)}\widehat{\cC}_\dg$$
in $\HO(\dgcat)(\Delta^{(r+1)})$, where $S.^{(r)}\widehat{\cC}_\dg$ is constant in one simplicial direction and $PS.^{(r+1)}\widehat{\cC}_\dg$ denotes the simplicial path object of $S.^{(r+1)}\widehat{\cC}_\dg$. As in the proof of \cite[Proposition~14.11]{Duke} we observe that the above sequence of morphisms is split exact at each component and that $PS.^{(r+1)}\widehat{\cC}_\dg$ is simplicially contractible. Hence, we obtain the following natural isomorphism
\begin{equation}\label{eq:nat}
\pi_!\, \Uaddu(S.^{(r+1)}\widehat{\cC}_\dg) \simeq \Sigma\left( \pi_!\, \Uaddu(S.^{(r)}\widehat{\cC}_\dg)\right)\,.
\end{equation}
By combining \eqref{eq:nat} with \eqref{eq:hyp} we conclude that
$\pi_!\, \Uaddu (S.^{(r+1)}\widehat{\cC}_\dg)$ is naturally isomorphic to $\Sigma^{(r+1)} \Uaddu(\widehat{\cC}_\dg)$, which achieves the proof.
\end{proof}
\begin{lemma}\label{lem:aux3}
The induced morphism $\pi_!\, \Uaddu(\Phi^{(n+m)})$ (see \eqref{eq:keymap}) becomes invertible in $\Maddu(e)$.
\end{lemma}
\begin{proof}
The derivator $\Maddu$ carries a homotopy colimit preserving symmetric monoidal structure and both morphisms $\Uaddu$ and $\varphi$ are symmetric monoidal. Therefore, the proof is a consequence of the following natural isomorphisms:
\begin{eqnarray}
\pi_!\, \Uaddu(S.^{(n)} \widehat{\cA}_\dg \otimes S.^{(m)}\widehat{\cB}_\dg) & \simeq & \pi_!\, \Uaddu(S.^{(n)} \widehat{\cA}_\dg \otimes^\bbL S.^{(m)}\widehat{\cB}_\dg) \label{eq:number6} \\
 & \simeq & 
\pi_!\, \Uaddu(S.^{(n)} \widehat{\cA}_\dg) \otimes \pi_!\, \Uaddu(S.^{(m)}\widehat{\cB}_\dg)  \label{eq:number1} \\
& \simeq & \Sigma^{(n)} \Uaddu(\widehat{\cA}_\dg) \otimes \Sigma^{(m)} \Uaddu(\widehat{\cB}_\dg) \label{eq:number2} \\
& \simeq &  \Sigma^{(n+m)} \Uaddu(\widehat{\cA}_\dg \otimes^\bbL \widehat{\cB}_\dg) \label{eq:number3} \\
& \simeq & \Sigma^{(n+m)} \Uaddu(\widehat{(\cA \otimes \cB)}_\dg) \label{eq:number4} \\
& \simeq & \pi_!\, \Uaddu(S.^{(n+m)}\widehat{(\cA \otimes \cB)}_\dg)\,. \label{eq:number5}
\end{eqnarray} 
Isomorphism \eqref{eq:number6} follows from the assumption that $\cA$ or $\cB$ (and hence $S.^{(n)}\widehat{\cA}_\dg$ or $S.^{(m)}\widehat{\cB}_\dg$) is $k$-flat. Isomorphisms \eqref{eq:number1} and \eqref{eq:number3} follow from the fact that the morphism $\Uaddu$ is symmetric monoidal and that the symmetric monoidal structure on $\Maddu(e)$ is homotopy colimit preserving; note that $\pi_!$ and $\Sigma$ are two (different) examples of homotopy colimits. Isomorphisms \eqref{eq:number2} and \eqref{eq:number5} follow from Lemma~\ref{lem:aux2}. Finally, isomorphism \eqref{eq:number4} follows from the natural derived Morita equivalences $\widehat{\cA}_\dg \otimes^\bbL \widehat{\cB}_\dg\simeq \widehat{\cA}_\dg \otimes \widehat{\cB}_\dg \simeq \cA \otimes \cB \simeq \widehat{(\cA\otimes \cB)}_\dg$. 
\end{proof}
Now, let us consider the following (composed) maps
\begin{equation}\label{eq:pair-aux}
\xymatrix@C=2em@R=1em{
 | \Map(\Uaddu(\uk), \pi_!\, \Uaddu(S.^{(n)} \widehat{\cA}_\dg))| \times | \Map(\Uaddu(\uk), \pi_!\, \Uaddu(S.^{(m)} \widehat{\cB}_\dg))|  \ar[d]   \\
 | \Map(\Uaddu(\uk), \pi_!\, \Uaddu(S.^{(n)} \widehat{\cA}_\dg \otimes S.^{(m)}\widehat{\cB}_\dg))| \ar[d]^-{\simeq} \\
| \Map(\Uaddu(\uk), \pi_!\, \Uaddu(S.^{(n+m)} \widehat{(\cA\otimes\cB)}_\dg))|
}
\end{equation}
The ``upper'' map is the one induced by the symmetric monoidal structure on $\Maddu$ and by the natural isomorphism
$$ \pi_!\, \Uaddu(S.^{(n)} \widehat{\cA}_\dg) \otimes \pi_!\, \Uaddu(S.^{(m)}\widehat{\cB}_\dg) \simeq \pi_!\, \Uaddu(S.^{(n)} \widehat{\cA}_\dg \otimes S.^{(m)}\widehat{\cB}_\dg)\,.$$
The ``lower'' map is the isomorphism induced by $\pi_!\, \Uaddu(\Phi^{(n+m)})$; see Lemma~\ref{lem:aux3}. A careful analysis of the proof of~\cite[Theorem~15.9]{Duke} show us that given an arbitrary dg category $\cC$ and an integer $r \geq 1$, we have natural weak equivalences
\begin{equation*}
K(\cC)_r\simeq \bbR\Hom(\Uadd(\uk), \Uadd(\cC))_r  \simeq  |\Map(\Uaddu(\uk), \pi_!\, \Uaddu (S.^{(r)}\widehat{\cC}_\dg))|\,.
\end{equation*}
Therefore, since the morphism $\varphi: \Maddu \to \Madd$ is symmetric monoidal, we conclude that the maps \eqref{eq:maps} associated to the pairing \eqref{eq:pairing3} agree with the above (composed) maps \eqref{eq:pair-aux}. Through a simple iteration of \cite[Proposition~14.12]{Duke} we obtain moreover a natural weak equivalence
$$ |\Map(\Uaddu(\uk), \pi_!\, \Uaddu (S.^{(r)} \widehat{\cC}_\dg))| \simeq \underset{q_1, \ldots, q_r}{\mathrm{hocolim}}\, | \Map(\uk, S^{(r)}_{q_1, \ldots , q_r}\widehat{\cC}_\dg)|\,.$$
This description, combined with the fact that the functor 
$$\Uaddu(e):\HO(\dgcat)(e)=\Hmo \too \Maddu(e)$$
is symmetric monoidal, allows us to conclude that the above (composed) maps \eqref{eq:pair-aux} can be expressed as the following homotopy colimits
$$ \underset{\substack{q_1, \ldots, q_n \\p_1, \ldots, p_m}}{\mathrm{hocolim}} \left( 
\begin{array}{c}
\underset{\downarrow}{|\Map(\uk, S^{(n)}_{q_1, \ldots, q_n}\widehat{\cA}_\dg)| \times | \Map(\uk, S^{(m)}_{p_1, \ldots, p_m}\widehat{\cB}_\dg)|} \\
\underset{\downarrow}{|\Map(\uk, S^{(n)}_{q_1, \ldots, q_n}\widehat{\cA}_\dg \otimes S^{(m)}_{p_1, \ldots, p_m}\widehat{\cB}_\dg)|} \\
 |\Map(\uk, S^{(n+m)}_{\substack{q_1, \ldots, q_n \\p_1, \ldots, p_m}}\widehat{(\cA\otimes \cB)}_\dg)| 
\end{array}
\right)\,.
$$
The ``upper'' map is the one induced by the symmetric monoidal structure on $\Hmo$. The ``lower'' map is the one obtained by applying the functor $|\Map(\underline{k},-)|$ to the dg functor \eqref{eq:keymaps}. Now, in order to conclude the proof it remains to show that this homotopy colimit agrees with the one described in \eqref{eq:colimit}. Note that given an arbitrary dg category $\cC$, the category $\cR(\uk,\cC)$ (see \S\ref{sub:Morita}) identifies naturally with $\widehat{\cC}$ which implies that $|\Map(\uk, \cC)| \simeq | w \widehat{\cC}|$. In the particular case of the dg categories $S^{(r)}_{q_1, \ldots, q_r}\widehat{\cC}_\dg$, the associated categories of perfect modules are naturally identified with $S^{(r)}_{q_1, \ldots, q_r}\widehat{\cC}$. Therefore $|\Map(\uk, S^{(r)}_{q_1, \ldots, q_r}\widehat{\cC}_\dg)| \simeq | w S^{(r)}_{q_1, \ldots, q_r}\widehat{\cC}|$ and so the maps in the above homotopy colimit are obtained by applying $|w-|$ to the composed functors
$$ S^{(n)}_{q_1, \ldots, q_n}\widehat{\cA} \times S^{(m)}_{p_1, \ldots, p_m}\widehat{\cB}  \too \widehat{S^{(n)}_{q_1, \ldots, q_n}\widehat{\cA}_\dg\otimes S^{(m)}_{p_1, \ldots, p_m}\widehat{\cB}_\dg} \too S^{(n+m)}_{\substack{q_1, \ldots, q_n \\p_1, \ldots, p_m}}\widehat{(\cA\otimes \cB)}\,.$$
The left-hand side functor is the one induced by the symmetric monoidal structure on $\Hmo$. The right-hand side functor is the classical extension of scalars associated to the dg functor \eqref{eq:keymaps}. Finally, a direct inspection show us that the above composed functors agree with the ones described in \eqref{eq:k-exact}. This implies that the above homotopy colimit agrees with the one described in \eqref{eq:colimit} and so the proof is finished.
\end{proof}
\section{Proof of Theorem~\ref{thm:2}}
\begin{proof}
We start by showing that $ch^k_n$ agrees with the $n^{\textrm{th}}$ Dennis trace map. Thanks to Keller-McCarthy's delooping theorem (see \cite[\S1.13]{cyclic} and \cite[Corollary~3.6.3]{McCarthy}) the spectrum $\bbR\Hom_{\cD(k)}(k,HH(\underline{A}))$ can be described as follows:
\begin{equation}\label{eq:spec-HH}
\Omega|HH_{DK}(S.\widehat{\uA}_\dg)| \quad |HH_{DK}(S.\widehat{\uA}_\dg)| \quad \cdots \quad |HH_{DK}(S^{(n)}.\widehat{\uA}_\dg)| \quad \cdots\,.
\end{equation}
Some explanations are in order. As in the proof of Theorem~\ref{thm:1}, $S.^{(n)}\widehat{\uA}_\dg$ can be considered as a multi-simplicial dg category in a natural way. By taking at each degree the simplicial $k$-module ($HH_{DK}$), which (via the Dold-Kan equivalence) corresponds to the Hochschild homology complex ($HH$), we obtain a multi-simplicial $k$-module $HH_{DK}(S.^{(n)}\widehat{\uA}_\dg)$ whose geometric realization we denote by $|HH_{DK}(S.^{(n)}\widehat{\uA}_\dg)|$. The structure maps of \eqref{eq:spec-HH} are induced by the fibration sequences
$$ |HH_{DK}(S.^{(n)}\widehat{\uA}_\dg)| \too |HH_{DK}(PS.^{(n)}\widehat{\uA}_\dg)| \too |HH_{DK}(S.^{(n+1)}\widehat{\uA}_\dg)| \,,$$
where the middle term is contractible; see \cite[Theorem~3.3.3]{McCarthy}. Now, since $K(\uA)$ (see \S\ref{sub:Kdg}) and \eqref{eq:spec-HH} are connective $\Omega$-spectra, it suffices to study the restriction of the map of spectra $ch^k$ (see \eqref{eq:diag-spt}) to its first component. Concretly, we need to study the induced map
\begin{equation}\label{eq:map-aux}
K(\uA)_1 \simeq \bbR\Hom(\Uadd(\uk), \Uadd(\uA))_1 \too | HH_{DK}(S.\widehat{\uA}_\dg)|\,.
\end{equation}
As explained in the proof of Theorem~\ref{thm:1}, we have natural weak equivalences
$$ \bbR\Hom(\Uadd(\uk), \Uadd(\uA))_1 \simeq |N.w S.\widehat{\uA}| \simeq | \Map(k, S.\widehat{\uA}_\dg)|\,.$$
On the other hand we have also the following natural weak equivalence
$$ |HH_{DK}(S.\widehat{\uA}_\dg)| \simeq | \Map_{\cD(k)}(k, HH(S.\widehat{\uA}_\dg))|\,.$$
Therefore, \eqref{eq:map-aux} is obtained by passing to the geometric realization of the map
\begin{equation}\label{eq:key3}
\Map(\underline{k},S.\widehat{\uA}_\dg) \too \Map_{\cD(k)}(k, HH(S.\widehat{\uA}_\dg))
\end{equation}
induced by the Hochschild homology functor. Now, recall from McCarthy~\cite{McCarthy} that associated with the $k$-algebra $A$ we have not only the category $\widehat{\uA}$ of perfect $A$-modules but also the exact category $\cP_A$ of finitely generated projective $A$-modules. As with any exact category, we can consider $\cP_A$ as a category with cofibrations and weak equivalences, where the weak equivalences are the isomorphisms. Moreover, since $\cP_A$ is $k$-linear we can also view it as a dg category. We have then an inclusion functor $\cP_A \hookrightarrow \widehat{\uA}$ and an inclusion dg functor $\cP_A \hookrightarrow \widehat{\uA}_\dg$. These functors allow us to construct the following (solid) diagram of bisimplicial sets:
$$
\xymatrix@C=2em@R=1.5em{
\mathrm{obj}(S.\cP_A) \ar@{-->}[rr]^M \ar[d]_\alpha^\simeq && HH_{DK}(S.\cP_A) \ar[dd]_\simeq^\gamma \\
N.iS.\cP_A \ar[d]_\beta^\simeq  & \\
N.wS. \widehat{\uA} \ar[rr]_-{\eqref{eq:key3}} && HH_{DK}(S.\widehat{\uA}_\dg)\,.
}
$$
The bisimplicial set $\mathrm{obj}(S.\cP_A)$ is constant in one simplicial direction and the morphism $\alpha$ corresponds to the inclusion of objects. As proved by Waldhausen \cite{Wald}, $\alpha$ is a weak equivalence since we are considering the simplicial category $iS.\cP_A$ of isomorphisms. The morphism $\beta$ is the one induced by the inclusion functor $\cP_A \hookrightarrow \widehat{\uA}$. As proved by Thomason-Trobaugh \cite[Theorem~1.11.7]{TT}, $\beta$ is also a weak equivalence since $\widehat{\uA}$ identifies naturally with the category of bounded cochain complexes in $\cP_A$. The morphism $\gamma$ is the one induced by the inclusion dg functor $\cP_A \hookrightarrow \widehat{\uA}_\dg$. Thanks to Keller-McCarthy's agreement property (see \cite[\S1.5]{cyclic} and \cite[Proposition~2.4.3]{McCarthy}), $\gamma$ is also a weak equivalence. Finally, $M$ is the induced morphism.

Now, let $q$ be a non-negative integer and $x$ an element of the set $\mathrm{obj}(S_q\cP_A)$. Under the weak equivalences $\alpha$ and $\beta$, this element corresponds to an object $x$ of the category $S_q\widehat{\uA}$. Using the natural dg enrichment $S_q\widehat{\uA}_\dg$ of $S_q\widehat{\uA}$, we can represent $x$ by a dg functor $\underline{x}: \uk \to S_q\widehat{\uA}_\dg$ which maps the unique object of $\uk$ to the object $x$. This dg functor is one of the $0$-simplices of the simplicial set $\Map(\uk, S_q\widehat{\uA}_\dg)$ (see \S\ref{sub:Morita}) and so it is mapped by the above map \eqref{eq:key3} to the following morphism of complexes 
$$ HH(\underline{x}):k \simeq HH(\uk) \too HH(S_q\widehat{\uA}_\dg)\,.$$
A direct inspection shows that this morphism corresponds to the zero cycle of $HH(S_q\widehat{\uA}_\dg)$ given by the identity of $x$, or equivalently to the $0$-simplice of the simplicial $k$-module $HH_{DK}(S_q\widehat{\uA}_\dg)$ given by the identity of $x$. Therefore, if we denote by $HH_{DK}(S.\cP_A)_0$ the $0$-simplices of $HH_{DK}(S.\cP_A)$, the morphism $M$ admits the following factorization
$$ \mathrm{obj}(S.\cP_A) \stackrel{\mathrm{id}}{\too} HH_{DK}(S.\cP_A)_0 \stackrel{\mathrm{inc}}{\hookrightarrow} HH_{DK}(S.\cP_A)$$
as in \cite[\S4.4]{McCarthy}. McCarthy has shown in \cite[\S4.5]{McCarthy} that this latter construction agrees with the Dennis trace map. Therefore, since the above morphisms $\alpha$, $\beta$ and $\gamma$ are weak equivalences, the proof concerning the Dennis trace map is finished.   

We now show that $ch_n^{\Lambda}$ agrees with the $n^{\textrm{th}}$ negative Chern character. Via the Dold-Kan correspondance between complexes and spectra, $\bbR\Hom_{\cD(\Lambda)}(k, C(\uA))$ identifies with the (desuspension of the) total negative complex associated to the simplicial mixed complex $C(S.\widehat{\uA}_\dg)$; see \cite[\S3]{McCarthy}. Following the same arguments as those of the Dennis trace map, we need to study the induced morphism of mixed complexes
\begin{equation*}
C(\underline{x}):k \simeq C(k) \too C(S_q\widehat{\uA}_\dg)\,.
\end{equation*}
As shown in \cite[\S2]{JK} this morphism of mixed complexes corresponds to the following morphism of complexes
\begin{equation}\label{eq:key-morph}
B^{-}C(\underline{x}): B^{-}(k) \too B^{-}(S_q\widehat{\uA}_\dg)\,,
\end{equation}
where $B^{-}$ denotes Connes' $B^{-}$-construction. Now, if we write $1$ for the unit of $k$, a simple computation show us that $B^{-}(k)$ is canonically endowed with the canonical zero cycle 
$$ \prod_{t=0}^{\infty} (-1)^t \frac{(2t)!}{t!}(1 \otimes \cdots \otimes 1) \,.$$
Therefore, since $1$ is mapped to the identity of $x$, the above morphism \eqref{eq:key-morph} corresponds to the following zero cycle
\begin{equation}\label{eq:cycle}
 \prod_{t=0}^{\infty} (-1)^t \frac{(2t)!}{t!}(\mathrm{id}_x \otimes \cdots \otimes \mathrm{id}_x) \in \mathsf{Z}_0 B^{-}(S.\widehat{\uA}_\dg)\,.
\end{equation}
We obtain then the same factorization 
$$ \mathrm{obj}(S.\cP_A) \stackrel{\eqref{eq:cycle}}{\too} \mathsf{Z}_0B^{-}(S.\widehat{\uA}_\dg) \stackrel{\mathrm{inc}}{\hookrightarrow} B^{-}(S.\widehat{\uA}_\dg)$$
of the morphism $M$ as in~\cite[\S4.4]{McCarthy}. McCarthy has shown in \cite[\S4.5]{McCarthy} that this latter construction agrees with the negative Chern character. Therefore, since the above morphisms $\alpha$, $\beta$ and $\gamma$ are weak equivalences, the proof is finished. 
\end{proof}
\section{Proof of Theorem~\ref{thm:3}}
\begin{proof}
The construction of the symmetric monoidal additive functor $N$ makes use of an auxiliar category $\Hmo_0$ and is divided in two steps:
$$ N : \BK \stackrel{N_1}{\too} \Hmo_0 \stackrel{N_2}{\too} \Madd(e)\,.$$
We start by describing the functor $N_2$. Recall from \cite{IMRN} that there is a symmetric monoidal functor
$$ \cU_a: \dgcat \too \Hmo \too \Hmo_0\,,$$
with values in an additive category which, heuristically, is the ``zero-dimensional'' analogue of the universal additive invariant $\Uadd$. The objects of $\Hmo_0$ are the dg categories and the (abelian group) morphisms $\Hom_{\Hmo_0}(\cB,\cA)$ are the Grothendieck group of the triangulated category $\rep(\cB,\cA)$; see \S\ref{sub:Morita}. Composition is induced by the (derived) tensor product of bimodules. The symmetric monoidal structure is induced by the (derived) tensor product of dg categories. Note that we have a natural symmetric monoidal functor $\Hmo \to \Hmo_0$ which sends an element $X$ of $\rep(\cB, \cA)$ to the corresponding class $[X]$ in the Grothendieck group $K_0\rep(\cB,\cA)$. A simple symmetric monoidal sharpening of \cite[Theorem~6.3]{IMRN} provides us the following universal characterization of $\cU_a$: given any additive category $\mathsf{A}$, endowed with a symmetric monoidal structure, we have an induced equivalence of categories
\begin{equation}\label{eq:univ-prop}
(\cU_a)^\ast: \Fun_{\,\mathrm{add}}^\otimes(\Hmo_0, \mathsf{A}) \stackrel{\sim}{\too} \Fun_{\,\mathrm{a}}^\otimes(\dgcat,\mathsf{A})\,,
\end{equation}  
here the left-hand side denotes the category of symmetric monoidal additive functors from $\Hmo_0$ to $A$ and the right-hand side the category of symmetric monoidal functors from $\dgcat$ to $\mathsf{A}$, which invert derived Morita equivalences and send split exact sequences to direct sums. Therefore, since the category $\Madd(e)$ is additive and the symmetric monoidal functor
$$ \Uadd(e): \Hmo=\HO(\dgcat)(e)  \too \Madd(e)$$
sends split exact sequences to direct sums, equivalence \eqref{eq:univ-prop} furnishes us a unique symmetric monoidal additive functor $N_2$ making the following diagram commute
\begin{equation}\label{eq:diagram5}
\xymatrix{
\Hmo \ar[d] \ar[dr]^{\Uadd(e)} & \\
\Hmo_0 \ar[r]_-{N_2} & \Madd(e)\,.
}
\end{equation}
Now, recall from \cite[\S4.7]{ICM} that a dg category $\cB$ is called {\em homotopically finitely presented} if for every filtered direct system of dg categories $\{\cC_j\}_{j \in J}$, the canonical map
$$ \mathrm{hocolim}_{j \in J} \Map(\cB, \cC_j) \too \Map(\cB, \mathrm{hocolim}_{j \in J}\, \cC_j)$$
is an isomorphism of simplicial sets. An important class of examples is provided by the {\em smooth} and {\em proper} dg categories in the sense of Kontsevich; see~\cite{IAS,Miami}. As proved in \cite[Theorem~15.10]{Duke}, equivalence \eqref{eq:corep-spt} and isomorphisms \eqref{eq:corep-K} can be greatly generalized: given dg categories $\cA$ and $\cB$, with $\cB$ homotopically finitely presented, we have a natural equivalence of spectra
\begin{equation}\label{eq:biv-ext1}
\bbR\Hom(\Uadd(\cB), \Uadd(\cA)) \simeq K\rep_\dg(\cB,\cA)
\end{equation}
and isomorphisms of abelian groups
\begin{equation}\label{eq:biv-ext2} 
\Hom(\Uadd(\cB),\Uadd(\cA)[-n]) \simeq K_n\rep_\dg(\cB,\cA) \qquad n \geq 0\,.
\end{equation}
Equivalence \eqref{eq:corep-spt}, resp. isomorphisms \eqref{eq:corep-K}, can be recovered from \eqref{eq:biv-ext1}, resp. from \eqref{eq:biv-ext2}, by taking $\cB=\uk$. Moreover, the category of perfect $\rep_\dg(\cB,\cA)$-modules is equivalent to $\rep(\cB,\cA)$ which implies that $K_0\rep_\dg(\cB,\cA)$ is naturally isomorphic to the Grothendieck group of the triangulated category $\rep(\cB,\cA)$. This allow us to conclude that the functor $N_2$ induces the following isomorphisms
\begin{equation}\label{eq:iso-induced}
\Hom_{\Hmo_0}(\cB,\cA) \stackrel{\sim}{\too} \Hom_{\Madd(e)}(N_2(\cB), N_2(\cA))\,.
\end{equation}
Now, recall that the mixed complex construction $\cA \mapsto C(\cA)$ inverts derived Morita equivalences, sends split exact sequences to direct sums, and is moreover symmetric monoidal. Hence, the above equivalence \eqref{eq:univ-prop} furnishes us a unique symmetric monoidal additive functor $\widetilde{C}: \Hmo_0 \to \cD(\Lambda)$ such that $\widetilde{C} \circ \cU_a =C$.
\begin{lemma}\label{lem:key4}
The following diagram commutes
\begin{equation*}
\xymatrix{
\Hmo_0 \ar[d]_{N_2} \ar[r]^{\widetilde{C}} & \cD(\Lambda) \\
\Madd(e) \ar[ur]_{\overline{C}(e)} & \,.
}
\end{equation*}
\begin{proof}
Both functors $\widetilde{C}$ and $\overline{C}(e) \circ N_2$ are symmetric monoidal and additive. Therefore, by pre-composing them with $\cU_a$ we obtain functors from $\dgcat$ to $\cD(\Lambda)$ which invert derived Morita equivalences and send split exact sequences to direct sums. Thanks to the above commutative diagram \eqref{eq:diagram5}, we observe that these latter functors are in fact the same, namely they are both the mixed complex construction. Making use of equivalence \eqref{eq:univ-prop} we then conclude that $\widetilde{C} = \overline{C}(e) \circ N_2$, which achieves the proof.
\end{proof}
\end{lemma}
We now describe the functor $N_1: \BK \to \Hmo_0$. Let us start with the additive category $\BK$. The objects of $\BK$ are the unital $k$-algebras and the (abelian group) morphisms $\Hom_{\BK}(B,A)$ are the Grothendieck groups of the exact categories $\Rep(B,A)$ of those $B\text{-}A$-bimodules which are projective and of finite type as $A$-modules; see \cite[\S II Definition~1.1]{Kassel-biv}. Composition is induced by the (derived) tensor product of bimodules; see \cite[\S II.2]{Kassel-biv}. The symmetric monoidal structure is the one induced by the tensor product of $k$-algebras; see \cite[\S II.3]{Kassel-biv}. Given unital $k$-algebras $A$ and $B$, we consider the natural functor
\begin{eqnarray*}
\Rep(B,A) \too \rep(\underline{B},\underline{A}) && P \mapsto \underline{P}\,,
\end{eqnarray*}
where $\underline{P}$ is concentrated in degree zero. This functor is fully-faithful and sends every short exact sequence of the exact category $\Rep(A,B)$ to a distinguished triangle of the triangulated category $\rep(\underline{B}, \underline{A})$. Therefore, it gives rise to an abelian group homomorphism
\begin{eqnarray}\label{eq:homo}
K_0\Rep(B,A) \too K_0\rep(\underline{B}, \underline{A}) && [P] \mapsto [\underline{P}]\,.
\end{eqnarray} 
\begin{lemma}\label{lem:Kassel-inv}
The abelian group homomorphism \eqref{eq:homo} is invertible.
\end{lemma}
\begin{proof}
Let $X$ be an object of $\rep(\underline{B}, \underline{A})$. Since this bimodule is compact as an object in $\cD(\underline{A})$, we can assume that it is given by a bounded complex of $B\text{-}A$-bimodules $(X_n)_{n \in \bbZ}$ which are projective and of finite type as $A$-modules. Hence, the assignment 
\begin{equation*}
[X] \longmapsto \sum_n (-1)^n [X_n] \in K_0 \Rep(B,A)
\end{equation*}
extends by linearity to $K_0\rep(\underline{B},\underline{A})$ and gives rise to the inverse of \eqref{eq:homo}.
\end{proof}
The above abelian group isomorphisms \eqref{eq:homo} are compatible with the composition operations as well as with the symmetric monoidal structures of the categories $\BK$ and $\Hmo_0$. Hence, they assemble themselves in a well-defined fully-faithful symmetric monoidal  additive functor $N_1: \BK \too \Hmo_0$. Moreover, the bivariant Chern character constructed by Kassel in \cite[\S II.4]{Kassel-biv} corresponds, under the above description of $\BK$, to a symmetric monoidal functor $ch_B: \BK \to \cD(\Lambda)$.
\begin{lemma}\label{lem:key3}
The following diagram commutes
\begin{equation*}
\xymatrix{
\BK \ar[d]_{N_1} \ar[r]^-{ch_B} & \cD(\Lambda) \\
\Hmo_0 \ar[ur]_{\widetilde{C}} & .
}
\end{equation*} 
\end{lemma} 
\begin{proof}
Let us start by recalling from \cite[\S II.4]{Kassel-biv} Kassel's construction of the bivariant Chern character. Let $A$ and $B$ be two unital $k$-algebras and $P$ an object of $\Rep(B,A)$. Since $P$ is projective and of finite type as a $A$-module, there is a $A$-module $Q$ and an isomorphism $\alpha: P\oplus Q \stackrel{\sim}{\to} A^n$. We can then consider the following morphism of $k$-algebras
\begin{equation*}
i: B \stackrel{\gamma}{\too} \mathrm{End}_A(P) \hookrightarrow \mathrm{End}_A(P\oplus Q) \mathop{\stackrel{\mathrm{Ad}(\alpha)}{\too}}_{\sim} M_n(A)\,. 
\end{equation*}
The morphism $\gamma$ is the one induced by the $B\text{-}A$-bimodule structure of $P$ and $\mathrm{Ad}(\alpha)$ denotes the conjugation by $\alpha$. Note that although each one of the $k$-algebras in the above morphism is unital, the inclusion morphism $\mathrm{End}_A(P) \hookrightarrow \mathrm{End}_A(P\oplus Q)$ does not preserve the unit. In particular, $i$ is not unit preserving. Nevertheless, there is a well-defined morphism between the associated mixed complexes. Kassel's bivariant Chern character $ch_B([P])$ of $P$ is by definition the composed morphism $\mathrm{Tr} \circ C(i)$ in $\cD(\Lambda)$, where $\mathrm{Tr}$ is the generalized trace map. In particular, this construction is independ of the choices of $Q$, $n$, and $\alpha$.

Now, recall from \cite[\S2.4]{Keller} that given unital $k$-algebras $A$ and $B$ and a (not necessarily unit preserving) homomorphism $\varphi:B \to A$, we can construct the $B\text{-}A$-bimodule $\varphi({\bf 1}_B)A_A$, whose $B\text{-}A$-action is given by $b\cdot \varphi(1)a\cdot a':=\varphi(b)aa'$. This $B\text{-}A$-bimodule gives rise to an object $\underline{\varphi({\bf 1}_B)A_A}$ of $\rep(\underline{B},\underline{A})$ and the value of the mixed complex construction on it agrees with the value of the mixed complex construction on $\varphi$. Consider then the following diagram
\begin{equation*}
\underline{B} \stackrel{\underline{\gamma}}{\too} \underline{\mathrm{End}_A(P)} \hookrightarrow \underline{\mathrm{End}_A(P\oplus Q)} \mathop{\stackrel{\underline{\mathrm{Ad}(\alpha)}}{\too}}_{\sim} \underline{M_n(A)} \stackrel{\underline{e_{11}}}{\longleftarrow} \underline{A} 
\end{equation*}
in $\Hmo$, where $e_{11}: A \to M_n(A)$ is the $k$-algebra homomorphism  which maps $A$ to the position $(1,1)$. Since $e_{11}$ induces an equivalence between $A$-modules and $M_n(A)$-modules, the morphism $\underline{e_{11}}$ becomes invertible in $\Hmo$. Moreover, the $\underline{B}\text{-}\underline{A}$-bimodule $\underline{P}$ can be expressed as the composition $\underline{e_{11}}^{-1} \circ \underline{i}$. Hence, by the above considerations, we conclude that the value of $\widetilde{C}$ on $[\underline{P}]$ is the composed morphism $C(\underline{e_{11}})^{-1} \circ C(\underline{i})$ in $\cD(\Lambda)$. Finally, since $C(\underline{e_{11}})$ and the generalized trace map $\mathrm{Tr}$ are inverse of each other (see \cite[\S3]{Kassel-biv}) we conclude that $ch_B([P]) = \widetilde{C}([\underline{P}])$. By linearity, this implies that $ch_B=\widetilde{C} \circ N_1$ and so the proof is finished. 
\end{proof}
We are now ready to conclude the proof of Theorem~\ref{thm:3}. Let $N$ be the composed functor $N_2 \circ N_1$. Then, the diagram of Theorem~\ref{thm:3} can be obtained by concatenating the commutative diagrams of Lemmas~ \ref{lem:key4} and \ref{lem:key3}. In particular, it is commutative. Recall from \cite[\S4.7]{ICM} that a $k$-algebra $B$ is homotopically finitely presented as a dg $k$-algebra if and only if the naturally associated dg category $\underline{B}$ is homotopically finitely presented as a dg category. Therefore, isomorphisms \eqref{isom-thm3} follow from \eqref{eq:iso-induced} and Lemma~\ref{lem:Kassel-inv}.
\end{proof}
\section{Proof of Proposition~\ref{prop:4}}
\begin{proof}
By construction, the additive motivator $\Madd$ is triangulated. Therefore, as shown in \cite[\S A.3]{CT}, it comes equipped with a canonical action
\begin{eqnarray*}
\HO(\Spt) \times \Madd \too \Madd && (E,X) \mapsto E \otimes X
\end{eqnarray*}
of the derivator associated to spectra. Given any spectrum $E$ and objects $X$ and $Y$ in $\Madd(e)$, the following natural weak equivalence of spectra holds
\begin{equation}\label{eq:weak-spt}
\bbR\Hom_{\Madd(e)}(E\otimes X,Y) \simeq \bbR\Hom_{\Ho(\Spt)}(E, \bbR\Hom_{\Madd(e)}(X,Y))\,.
\end{equation}
Let us now consider the distinguished triangle of spectra
\begin{equation}\label{eq:classical-triangle}
 \bbS \stackrel{\cdot l}{\too} \bbS \too \bbS/l \too \bbS[1]\,,
\end{equation} 
where $\bbS$ is the sphere spectrum and $\cdot l$ the $l$-fold multiple of the identity morphism. By applying the functor $-\otimes \Uadd(\underline{\bbZ})$ to it we obtain a distinguished triangle in the category of non-commutative motives
$$ \Uadd(\underline{\bbZ}) \stackrel{\cdot l}{\too} \Uadd(\underline{\bbZ}) \too \bbS/l \otimes \Uadd(\underline{\bbZ}) \too \Uadd(\underline{\bbZ})[1]\,,$$
which allow us to conclude that $\bbS/l \otimes \Uadd(\underline{\bbZ})$ is the mod-$l$ Moore object $\Uadd(\underline{\bbZ})/l$ of $\Uadd(\underline{\bbZ})$. Therefore, by combining \eqref{eq:weak-spt} with \eqref{eq:corep-spt} we obtain the weak equivalence
$$ \bbR\Hom(\Uadd(\underline{\bbZ})/l, \Uadd(\cA)) \simeq \bbR \Hom_{\Ho(\Spt)}(\bbS/l, K(\cA))\,.$$
Note that by construction $\bbS/l$ is a dualizable object in the symmetric monoidal category $\Ho(\Spt)$. Its dual $(\bbS/l)^\vee$ is given by $\bbR\Hom_{\Ho(\Spt)}(\bbS/l, \bbS)$. By applying the duality functor $\bbR\Hom_{\Ho(\Spt)}(-,\bbS)$ to \eqref{eq:classical-triangle} we obtain the distinguished triangle
$$ \bbS[-1] \too (\bbS/l)^\vee \too \bbS \stackrel{\cdot l}{\too} \bbS \,,$$
and so by ``rotating it'' twice (see \cite[\S1.1]{Neeman}) we conclude that $(\bbS/l)^\vee$ identifies with $(\bbS/l)[-1]$. Therefore, we have the following weak equivalences
$$ \bbR\Hom_{\Ho(\Spt)}(\bbS/l, K(\cA)) \simeq (\bbS/l)^\vee \wedge^\bbL K(\cA) \simeq(\bbS/l \wedge^\bbL  K(\cA))[-1]\,. $$
Finally, since by definition $K(\cA; \bbZ/l) = \bbS/l \wedge^\bbL K(\cA)$ we have
\begin{equation*}
\bbR \Hom\left(\Uadd(\underline{\bbZ})/l, \Uadd(\cA) \right) \simeq K(\cA;\bbZ/l)[-1]\,.
\end{equation*}
Now, let us assume that $l=pq$ with $p$ and $q$ coprime integers. In this case, $\bbZ/l\simeq \bbZ/p \times \bbZ/q$ and so  the natural maps $\bbS/l \to \bbS/p$ and $\bbS/l \to \bbS/q$ induce an isomorphism $\bbS/l \simeq \bbS/p \vee \bbS/q$ in $\Ho(\Spt)$. Therefore, since we have the natural identifications
\begin{eqnarray*}
\Uadd(\underline{\bbZ})/l \simeq \bbS/l \otimes \Uadd(\underline{\bbZ}) & \Uadd(\underline{\bbZ})/p \simeq \bbS/p \otimes \Uadd(\underline{\bbZ})  & \Uadd(\underline{\bbZ})/q \simeq \bbS/q \otimes \Uadd(\underline{\bbZ})\,,
\end{eqnarray*}
and the functor 
\begin{eqnarray*}
\Ho(\Spt) \too \Madd(e) && E \mapsto E  \otimes \Uadd(\underline{\bbZ})
\end{eqnarray*}
preserves sums, we obtain the isomorphism $\Uadd(\underline{\bbZ})/l \simeq \Uadd(\underline{\bbZ})/p \oplus \Uadd(\underline{\bbZ})/q$ in $\Madd(e)$.
\end{proof}

\end{document}